\documentclass[a4paper,11pt,reqno]{article}
\usepackage{amssymb}
\usepackage{enumitem}
\usepackage{amsthm}
\usepackage[alphabetic]{amsrefs}
\usepackage{tikz}
\usepackage{blindtext, xcolor}
\usepackage{comment}
\usepackage{mathrsfs}
\usepackage{a4wide}
\usepackage{fullpage}
\usepackage{mathtools}
\usepackage{hyperref}

\usepackage{bm}

\newtheorem{corollary}{Corollary}[section]
\newtheorem{theorem}[corollary]{Theorem}
\newtheorem{lemma}[corollary]{Lemma}
\newtheorem{proposition}[corollary]{Proposition}

\newtheorem{remark}[corollary]{Remark}

\numberwithin{equation}{section}

\usepackage{graphicx} 

\title{Approximating the volume of a truncated relaxation of the independence polytope}

\author{Ferenc Bencs\footnote{Centrum Wiskunde \& Informatica, P.O. Box 94079 1090 GB Amsterdam, The Netherlands. Email: \texttt{ferenc.bencs@gmail.com}. Funded by the Netherlands Organisation of Scientific Research (NWO): VI.Veni.222.303}, 
Guus Regts\footnote{Korteweg de Vries Institute for Mathematics, University of Amsterdam. P.O. Box 94248 1090 GE Amsterdam  The Netherlands. Email: \texttt{guusregts@gmail.com}. Funded by the Netherlands Organisation of Scientific Research (NWO): VI.Vidi.193.068}}

\date{\today}

\begin{document}

\maketitle

\begin{abstract}
Answering a question of Gamarnik and Smedira~\cite{gamarnik2023computing}, we give a polynomial time algorithm that approximately computes the volume of a truncation of a relaxation of the independent set polytope, improving on their quasi-polynomial time algorithm.
Our algorithm is obtained by viewing the volume as an evaluation of a graph polynomial and we approximate this evaluation using Barvinok's interpolation method.
\end{abstract}

\section{Introduction}
Approximately computing the volume of a polytope and more generally a convex body is one of the success stories of randomized algorithms.
Indeed, Dyer, Frieze and Ravi~\cite{randomvolume} showed that with access to a membership oracle there is a randomized polynomial time algorithm to approximately compute the volume of a convex body, whereas no such deterministic algorithm can exist by a result of B\'ar\'any and F\"uredi~\cite{volumeisdifficult}.

Partly motivated by the question whether randomized algorithms are superior over deterministic algorithms in the context of approximately computing volumes of polytopes, Gamarnik and Smedira~\cite{gamarnik2023computing} considered the question of approximately computing the volume of a standard relaxation of the independent set polytope of a graph $G=(V,E)$ which is defined as
\[
P_G:=\{x\in [0,1]^V\mid x_u+x_v\leq 1 \text{ for all } uv\in E\}.
\]
In fact for a bipartite graph $G$, $P_G$ is equal to the independent set polytope of $G$~\cite{Schrijver}*{Theorem 19.7}. 
Since $P_G$ is given by inequality constraints, testing membership in $P_G$ can be done efficiently and hence there is an efficient randomized algorithm for approximately computing its volume~\cite{randomvolume}.
As far as we know, no efficient deterministic algorithm is known for this task for bounded degree graphs. For related result when the girth of $G$ tends to infinity see \cite{GamarnikKavitacontinuous}.

The choice of polytope $P_{G}$ may seem a bit arbitrary, but not much appears to be known about deterministic approximation of the volume of polytopes (see~\cites{Barvinoktransportationpolytope,barvinok2021quick} for deterministic approximations of the volume of transportation polytopes) and as such it is as good a choice as any other. 
Moreover, since integration can be seen as an idealized form of counting~\cite{gamarnik2023computing}, the polytope $P_G$ actually presents itself as a rather natural candidate to test our ability to design efficient deterministic approximation algorithms for computing volumes. 
Indeed, in the `counting world' it is known that for the problem of approximately computing the number of independent sets in a bounded degree graph (and more generally approximately computing evaluations of its independence polynomial) randomized algorithms doe not outperform deterministic algorithms in the sense that for graphs of maximum degree at most $5$ both efficient deterministic and randomized (Markov chain based) polynomial time algorithms exist for this task~\cites{weitz,spectral-hardcore}, while if the maximum degree is at least $6$ approximating the number of independent sets is \textsc{NP}-hard~\cites{slysun,GSVhardness}\footnote{This transition in the computational complexity of approximate counting in fact extends to evaluating the independence polynomial on bounded degree graphs.}.
The Markov chain based randomized algorithms do tend to have faster running times compared to their deterministic counterparts~\cites{spectral-hardcore,spectraloptimal,weitz,PR19}.
Incidentally the randomized algorithm for approximating volumes of Fyer Frieze and Ravi~\cite{randomvolume} is also based on Markov chains and as such it is natural to try to use techniques from approximate counting to design deterministic approximation algorithms in the context of volume approximation.

Gamarnik and Smedira~\cite{gamarnik2023computing} in fact considered a truncation of $P_G$, defined as follows.
For $\delta\in [0,1/2]$ and a graph $G=(V,E)$ let
\begin{equation}
P_{G,\delta}:=\{(x_v)_{v\in V}\in [0,1/2+\delta]^V\mid x_u+x_v\leq 1 \text{ for all } uv\in E \}   
\end{equation}
and note that $P_{G,1/2}$ is equal to $P_G$.
Using the successful correlation decay method from approximate counting~\cites{weitz,BGcountingwithoutsmapling} Gamarnik and Smedira~\cite{gamarnik2023computing} managed to design an algorithm that approximately compute the volume of a $P_{G,\delta}$ for $\delta=O(1/\Delta(G)^2)$ in time $n^{O(\log n)}$, where $n$ denotes the number of vertices of $G$.

Our main result improves on their result, both in terms running time as well as in terms of the bound on $\delta.$

\begin{theorem}\label{thm:main}
Let $\Delta>0$. There exists a constant $C>0$ such that for each and $\delta\in [0,C/\Delta]$ there exists an algorithm that on input of an $n$-vertex graph of maximum degree at most $\Delta$ and $\varepsilon\in (0,1)$ computes a number $\xi$ such that
\[
(1-\varepsilon)\xi\leq \mathrm{vol}(P_{G,\delta})\leq (1+\varepsilon)\xi
\]
in time polynomial in $n/\varepsilon$.
\end{theorem}

While we can provide an explicit bound on the constant $C$ in the theorem above, but we choose not to since ideally we would like to be able to take $\delta$ to be $1/2$ and this appears to be out of reach for now. Indeed, our approach to proving Theorem~\ref{thm:main} consists of viewing the volume of $P_{G,\delta}$ as the evaluation of a certain graph polynomial/partition function and by applying Barvinok's interpolation method~\cites{barbook, PR17} to it. 
For this we need to show the polynomial is zero-free on some open region of the complex plane and and we manage to do this for $\delta=O(1/\Delta)$. 
It appears to be unlikely that this can be pushed beyond $O(1/\Delta)$; we will say more about this in Section~\ref{sec:remarks}.

The remainder of this paper is organized as follows.
In the next section we will indicate how to interpret the volume of $P_{G,\delta}$ as the evaluation of a polynomial that is naturally associated to $G$ and indicate how this leads to a proof of Theorem~\ref{thm:main} provided two conditions are met.
In Sections~\ref{sec:zero free} and~\ref{sec:coefficients} we verify these conditions.
Finally, in Section~\ref{sec:remarks} we conclude with some discussion and possible extensions.

\section{The volume as an evaluation of a graph polynomial}\label{sec:forest gen}
In this section we will show how to view $\text{Vol}(P_{G,\delta})$ as the evaluation of a certain forest generating function and explain how this allows us to employ Barvinok's interpolation method to prove Theorem~\ref{thm:main}.
To this end we introduce some terminology.

For a total order on the edges of a graph $H=(V,E)$ and a spanning tree $(V,T)$, $T\subset E$, we call $e\in E\setminus T$ \emph{broken} if $T\cup \{e\}$ contains a cycle in which $e$ is the largest edge.
We denote by $E_T$ the collection of broken edges.  A well known combinatorial lemma, going back to Penrose~\cite{penrose1967convergence}, states the following:
\begin{lemma}\label{lem:penrose}
Let $H=(V,E)$ be a graph with a fixed total ordering of the edges.
Then the collection of connected subgraphs of $H$ can be partitioned into intervals of the form $\{F\mid T\subseteq F \subseteq T\cup E_T\}$, where $T$ is a tree.
\end{lemma}
\begin{remark}
A simple proof of Lemma~\ref{lem:penrose} can be obtained by associating a connected subgraph $(V,F)$ of $H$ to its minimum weight spanning tree.    
\end{remark}

We start by rewriting the volume of $P_{G,\delta}$. 
We have
\begin{equation}\label{eq:vol with indicator}
\text{Vol}(P_{G,\delta})=\int \prod_{uv\in E} \mathbf{1}_{x_u+x_v\leq 1} d\mu,
\end{equation}
where $\mu$ denotes the Lebesgue measure, where we integrate over $[0,1/2+\delta]^V$ and where $\mathbf{1}_{x_u+x_v\leq 1}$ denotes the function that on input of $\bm{x}\in \mathbb{R}^V$ outputs $1$ if $x_u+x_v\leq 1$ and zero otherwise.
In what follows it will be convenient to denote $I=I_\delta:=[0,1/2+\delta]$.
Note that we can write
\[
\mathbf{1}_{x_u+x_v\leq 1}=1-\mathbf{1}_{x_u+x_v>1}
\]
and hence
\begin{equation}\label{eq:expand 1}
\prod_{uv\in E} \mathbf{1}_{x_u+x_v\leq 1}=\sum_{A\subseteq E}(-1)^{|A|}\prod_{uv\in A}\mathbf{1}_{x_u+x_v> 1}.
\end{equation}
We can expand this further by identifying the connected components of $A$. To this end denote for $k\in \mathbb{Z}_{\geq 0}$, by $\pi_k(V)$ the collection of partitions of $V$ into $k$ nonempty sets.
We have that~\eqref{eq:vol with indicator} is equal to
\begin{align}\nonumber
&\int_{I^V}\sum_{k\geq 0}\sum_{P\in \pi_k(V)} \prod_{\substack{S\in P\\ |S|\geq 2}} \sum_{\substack{A\subseteq E(S)\\ (S,A) \text{ connected}}} (-1)^{|A|}\prod_{uv\in A}\mathbf{1}_{x_u+x_v> 1}d\mu
\\
=&\sum_{k\geq 0}\sum_{P\in \pi_k(V)} \prod_{\substack{S\in P\\ |S|\geq 2}} \int_{I^S}\sum_{\substack{A\subseteq E(S)\\ (S,A) \text{ connected}}} (-1)^{|A|}\prod_{uv\in A}\mathbf{1}_{x_u+x_v> 1}d\mu \cdot \prod_{\substack{P\in \pi_k(V)\\ |P|=1}} \int_I 1 d\mu.\label{eq:expand partition}
\end{align}
Now focusing on the contribution of a single set $S\in P$ of size at least $2$, denoting $E(S)$ for the edge set of $G[S]$, and fixing some aribtrary ordering of edges of $G$, we have
\begin{align}
&\sum_{\substack{A\subseteq E(S)\\ (S,A) \text{ connected}}} (-1)^{|A|}\prod_{uv\in A}\mathbf{1}_{x_u+x_v> 1}\nonumber
\\
=&\sum_{\substack{T\subseteq E(S)\\ (S,T) \text{ tree}}}\sum_{A\subseteq E(S)_T } (-1)^{|T\cup A|}\prod_{uv\in T\cup A}\mathbf{1}_{x_u+x_v> 1} \nonumber
\\
=&\sum_{\substack{T\subseteq E(S)\\ (S,T) \text{ tree}}}(-1)^{|T|}\prod_{uv\in T}\mathbf{1}_{x_u+x_v>1}\cdot \prod_{st\in E(S)_T}\mathbf{1}_{x_{s}+x_{t}\leq 1},\label{eq:tree}
\end{align}
where $E(S)_T\subseteq E(S)$ denotes the set of broken edges associated to $T$.

Define for an induced subgraph $H=(S,E(S))$ of $G$,
\[
w(H):=\frac{1}{(1/2+\delta)^{|S|}}\sum_{\substack{T\subseteq E(S)\\ (S,T) \text{ tree}}}(-1)^{|T|}\int_{I^S}\prod_{uv\in T}\mathbf{1}_{x_u+x_v>1}\cdot \prod_{st\in E(S)_T}\mathbf{1}_{x_{s}+x_{t}\leq 1} d\mu.
\]
We note that since the first line of~\eqref{eq:tree} does not depend on the ordering of $E$ this quantity really only depends on the graph $H$. 
We moreover note that in case $H$ is not connected we have $w(H)=0$ and if $H$ consists of a single vertex we have $w(H)=1$.

Let us next define the univariate polynomial $p_{G,\delta}(x)$, for an $n$-vertex graph $G$, by
\begin{equation}\label{eq:def p}
p_{G,\delta}(x):=\sum_{k\geq 0}\sum_{P\in \pi_{n-k}(V)}\prod_{S\in P} w(G[S]) x^k,
\end{equation}

Using that $\int_{I}1 d\mu=1/2+\delta$, we then arrive at the following conclusion.
\begin{proposition}\label{prop: vol= forest gen}
With notation as above we have
\begin{equation}\label{eq:vol = forest gen}
 (1/2+\delta)^{-|V|}\text{Vol}(P_{G,\delta})=p_{G,\delta}(1). 
\end{equation}
\end{proposition}

The upshot of this proposition is that approximating the volume of $P_{G,\delta}$ is equivalent to approximating the polynomial $p_{G,\delta}(x)$ evaluated at $1$.
We will do this using Barvinoks's interpolation method~\cite{barbook} combined with the refinement due to Patel and the second author~\cite{PR17}.
For this it is convenient to express the coefficients of $p_{G,\delta}$ as induced subgraph counts.
More precisely, for a graph $H$, $\text{ind}(H,G)$ denotes the number of subsets $S$ of $V(G)$ such that $G[S]$ is isomorphic to $H$.  
We note that for $k\in \mathbb{Z}_{\geq 0}$ the coefficient of $x^k$ in $P_G$ is equal to 
\begin{equation}\label{eq:coef x^k}
    \sum_{\substack{H\in \mathcal{G}^*\\ |V(H)|=k+k(H)}} \text{ind}(H,G) \prod_{C \text{ comp. of } H} w(C), 
\end{equation}
where $k(H)$ denotes the number of components of $H$ and where $ \mathcal{G}^*$ denotes the family of graphs in which each component has at least two vertices.
This implies that we can write $p_{G,\delta}(x)=\sum_{k\geq 0}e_k x^k$, where $e_k$ is of the form $e_k=e_k(G)=\sum_{H}\lambda_{H,k}\text{ind}(H,G)$ for certain numbers $\lambda_{H,k}$ that only depend on $H$ and $k$ and not on $G$.
We refer to $\lambda_{H,k}$ as the coefficient of $\text{ind}(H,\cdot)$ in $p_{G,\delta}$.


To be able to apply the interpolation method we then need the following ingredients. Here $\mathcal{G}_\Delta$ denotes the family of all graphs of maximum degree at most $\Delta$.
\begin{itemize}
    \item A region $U$ in the complex plane that contains $0$ and $1$ such that for all $z\in U$ and all $G\in \mathcal{G}_\Delta$, $p_{G,\delta}(z)\neq 0$, and,
    \item an algorithm that on input of an $m$-vertex graph $H=(V,E)\in \mathcal{G}_\Delta$ computes the coefficient $\lambda_{H,k}$ of $\text{ind}(H,\cdot)$ in $p_{G,\delta}$ in time $C^{O(m)}$ for a constant $C>0$.
\end{itemize}
In the next section we will prove the following result, taking care of the first ingredient.
\begin{theorem}\label{thm: zero-free disk}
Let $\Delta>0$. Then there exist constants $C\in (0,1)$ and $K>1$ such that if $\delta<C/\Delta$ and $|x|<K$, then $p_{G,\delta}(x)\neq 0$ for all graphs $G$ of maximum degree at most $\Delta$.
\end{theorem}
For experts in the use of statistical physics methods such as the Koteck\'y-Preiss conditions~\cite{KP86}, the Dobrushin conditions~\cite{dobrushin1996estimates} or Gruber-Kunz conditions~\cites{GruberKunz,BFPgruberkunzbound}, this should not be a difficult exercise. 
For completeness we will however provide a detailed and self-contained proof below inspired by~\cite{JPRchromatic}.

In the subsequent section we will prove the following result, addressing the second ingredient.
\begin{theorem}\label{thm:algorithm}
Let $\Delta>0$. There exists an algorithm that on input of a $m$-vertex graph $G=(V,E)\in \mathcal{G}_\Delta$ computes the coefficient $\lambda_{H,k}$ of $\text{ind}(H,\cdot)$ in $p_{G,\delta}$ in time $\Delta^{O(k)}$, where the implicit constants do not depend on $\Delta$.
\end{theorem}

Combining these two ingredients with Theorem 3.2 from~\cite{PR17} we immediately obtain the desired deterministic algorithm to approximate the volume of the polytope $P_{G,\delta}$, thereby proving Theorem~\ref{thm:main}.

\section{A zero-free disk for $p_{G,\delta}$}\label{sec:zero free}
In this section we prove the following result implying Theorem~\ref{thm: zero-free disk}.
Here $e$ is the base of the natural logarithm.
\begin{theorem}\label{thm: zero-free  precise}
Let $\Delta>1$. Suppose $K>1$ and $\delta\in (0,1)$ satisfy
\begin{equation}\label{eq:condition K and delta}
4\delta K \leq \frac{\log(a)(1-1/\Delta)}{a\Delta}
\end{equation}
for some $a\in (1,e)$.
Then for any graph $G$ of maximum degree at most $\Delta$ and $x\in \mathbb{C}$ such that $|x|\leq K$, $p_{G,\delta}(x)\neq 0$.
\end{theorem}

First of all it will be convenient to reformulate $p_{G,\delta}(x)$ as a certain forest generating function.
By first summing over the forests and noting that this induces a partition of the vertex set we see that~\eqref{eq:expand partition} is equal to
\[
\sum_{\substack{F\subseteq E\\ F \text{ forest}}} (-1)^{|F|}\prod_{T \text{ comp. of }F} \int_{I^{V(T)}}\prod_{uv\in T}\mathbf{1}_{x_u+x_v>1}\cdot \prod_{st\in E_T}\mathbf{1}_{x_{s}+x_{t}\leq 1}d\mu,
\]
where we insist that $E_T$ only depends on $T$ and the graph induced by $V(T)$.
We now introduce for a tree $T\subseteq E$ (not necessarily a spanning tree) its weight
\begin{equation}\label{eq:def weight}
    w_T:=(1+\delta)^{-|V(T)|}\int_{I^{V(T)}} (-1)^{|T|}\prod_{uv\in T}\mathbf{1}_{x_u+x_v>1}\cdot \prod_{st\in E_T}\mathbf{1}_{x_{s}+x_{t}\leq 1} d\mu,
\end{equation}
It then follows that 
\begin{equation}\label{eq:rewrite to forest}
p_{G,\delta}(x)=\sum_{\substack{F\subseteq E(G)\\F \text{ forest}}} x^{|F|}\prod_{T \text{ comp. of }F} w_T,
\end{equation}
an expression of $p_{G,\delta}$ as some kind of forest generating function.

Now we have the desired reformulation we require some preliminary results before we can prove Theorem~\ref{thm: zero-free  precise}.
Let for a graph $G$ and a vertex $v\in V(G)$, $T_{G,v}(x)$, denote the rooted tree generating function of $(G,v)$, i.e.,
\[
T_{G,v}(x)=\sum_{\substack{T\subseteq E(G)\\ v\in V(T)}}x^{|T|},
\]
where the sum is over trees $T$.
We require the following simple lemma, found as Lemma 4.6 in~\cite{PRsurvey}. 
\begin{lemma}\label{lem:tree bound}
Let $\Delta>0$ and let $G$ be a graph of maximum degree at most $\Delta$.
Then, for any $a>1$, \[T_{G,v}\left(\tfrac{\log a}{a\Delta}\right)\leq a.\]
\end{lemma}
The next lemma says we can bound the weights $w_T$.
\begin{lemma}\label{lem:weight bound}
Let $G=(V,E)$ be a graph and let $\delta\in (0,1/2)$.
For any tree $T\subseteq E$ with at least one edge we have
\[
|w_T|\leq \left(\frac{2\delta}{1/2+\delta}\right)^{|V(T)|}.
\]
\end{lemma}
\begin{proof}
Let us denote the number of vertices of $T$ by $k\geq 2$.
The denominator in the definition of $w_T$ (as found in~\eqref{eq:def weight}) is equal to $(1/2+\delta)^{k}$.
So it suffices to bound the numerator.
The numerator can simply be bounded by 
\[
\int_{[0,1/2+\delta]^{V(T)}} \prod_{uv\in T}{\mathbf 1}_{x_u+x_v>1} d\mu,
\]
which in turn can be bounded by $(2\delta)^k$ as it is equal to the volume of a polytope contained in the box $[1/2-\delta,1/2+\delta]^{V(T)}$.
This proves the lemma.
\end{proof}

\begin{proof}[Proof of Theorem~\ref{thm: zero-free  precise}]
We will prove the theorem by proving inductively on the number of vertices of $G$ the following two statements:
\begin{itemize}
    \item $p_{G,\delta}(x)\neq 0$,
    \item for any vertex $v$ of $G$, $\left|\frac{p_{G,\delta}(x)}{p_{G-v,\delta}(x)}-1\right|\leq 1/\Delta$.
\end{itemize}
In case that $G$ consist of a single vertex the two statements are clearly true.
Next assume that $|V(G)|>1$.
Since $1/\Delta\in (0,1)$, the second item implies the first by induction. 
So we focus on proving the second item.

We have for any vertex $v$,
\[
p_{G,\delta}(x)=p_{G-v,\delta}(x)+\sum_{\substack{T\subseteq E\\ v\in V(T)}} x^{|T|}w_T p_{G\setminus V(T),\delta}(x),
\]
where the sum is over trees $T$ with at least one edge, containing the vertex $v.$
This follows by considering the forests that have a nontrivial component that contains $v$ and those that do not.
Since by induction $p_{G-v,\delta}(x)\neq 0$, this implies that 
\[
\frac{p_{G,\delta}(x)}{p_{G-v,\delta}(x)}-1=\sum_{\substack{T\subseteq E\\ v\in V(T)}} x^{|T|}w_T \frac{p_{G\setminus V(T),\delta}(x)}{p_{G-v,\delta}(x)}.
\]
By induction and a telescoping argument we can bound
\[
\left|\frac{p_{G\setminus V(T),\delta}(x)}{p_{G-v,\delta}(x)}\right|\leq \left(\frac{1}{1-1/\Delta}\right)^{|E(T)|}.
\]
Now simply bounding $\frac{2\delta}{1/2+\delta}$ by $4\delta$ we have by Lemma~\ref{lem:weight bound} that $|w_T|\leq (4\delta)^{|V(T)|}$ and therefore
\[
\left|\frac{p_{G,\delta}(x)}{p_{G-v,\delta}(x)}-1\right|\leq 4\delta (T_{G,v}(|x|\tfrac{4\delta}{1-1/\Delta})-1).
\]
Next we use that $|x|\delta \leq \frac{\log(a)(1-1/\Delta)}{a 4\Delta}$ by assumption, in combination with 
Lemma~\ref{lem:tree bound} to conclude that  
\[
\left|\frac{p_{G,\delta}(x)}{p_{G-v,\delta}(x)}-1\right|\leq 4\delta (T_{G,v}(\tfrac{\log a}{a\Delta})-1)\leq 4\delta(a-1)< 1/\Delta,
\]
as desired.
\end{proof}

\section{Computing the coefficients of $\text{ind}(H,\cdot)$}\label{sec:coefficients}
The aim in this section is to give a proof of Theorem~\ref{thm:algorithm}.
We recall that for a graph $H$ and $k\in \mathbb{N}$ such that $|V(H)|=k(H)+k$ we have that the coefficient $\lambda_{H,k}$ of $\text{ind}(H,G)$ is given by (cf.~\eqref{eq:coef x^k}),
\[
\lambda_{H,k}=\prod_{C \text{ comp. of } H} w(H), 
\]
provided each component of $H$ has size at least $2$; otherwise $\lambda_{H,k}=0$.
We use~\eqref{eq:rewrite to forest} to obtain the following alternative expression:
\begin{equation}\label{eq:coef ind}
\lambda_{H,k}=\sum_{\substack{F\subseteq E(H)\\ \text{ sp. forest}\\|F|=k}}\prod_{T \text{ comp. of }F} w_T.
\end{equation}
In the next subsection we will outline how to compute the $w_{T}$, after which we will give our proof of Theorem~\ref{thm:algorithm}.

\subsection{Computation of $w_T$}
Let us recall that 
\[
    w_T=\frac{(-1)^{|T|}\int \prod_{uv\in T}\mathbf{1}_{x_u+x_v>1}\cdot\prod_{st\in E_T}\mathbf{1}_{x_u+x_v\le 1}~d\mu}{\int 1~d\mu},
\]
where the integrations are over $I^{V(T)}$. 

The denominator in $w_T$ is equal to the volume of $I^{V(T)}$, that is $(1/2+\delta)^{|V(T)|}$. Thus we only have to focus on the integral appearing in the numerator, which we will denote by $(-1)^{|V(T)|}\hat w_T$. 

\begin{remark}
We note that $\hat{w}_T$ can be seen as the volume of a polytope with at most $\binom{\Delta k}{k}$ many vertices and as such the geometric algorithm (based on determinants) of Lawrence~\cite{Lawrencevolume} could be used to compute $\hat{w}_T$ in time $\Delta
^{O(k)}$.
We will however provide a direct/combinatorial algorithm.
\end{remark}
\begin{proposition} \label{prop:volume}
    For any connected graph $G=(V,E)$ on $k\ge 2$ vertices of maximum degree $\Delta$ and any spanning tree $T\subseteq E(G)$ the quantity
    \[
        \hat w_{T}=\int\prod_{uv\in T}\mathbf{1}_{x_u+x_v>1}\cdot\prod_{st\in E_ T}\mathbf{1}_{x_u+x_v\le 1}~d\mu
    \]
    can be computed in $\Delta^{O(k)}$ time, where the implicit constant does not depend on $\Delta$.
\end{proposition}

The remainder of this subsection is devoted to proving Proposition~\ref{prop:volume}.
Let us denote the function appearing in the integral as $f(\bm{x})$. To prove the proposition, we will reduce the problem to calculating the integral of $f(\bm{x})$ over some nice (but exponentially many) smaller polytopes.


First we have to compute $E_T$. It is clear from the definition that this can be done in polynomial time in the number of vertices, i.e. $\textsc{poly}(k)$.


Since all the variables appearing in the integral have to be in the interval $[0,1/2+\delta]$ and $T$ is connected, it means that any variable has to be in $J=[1/2-\delta,1/2+\delta]$, thus we can restrict the integral of $f(\bm{x})$ over $J^{V}$. 
For later purposes let $J_-=[1/2-\delta,1/2]$.

For each $S\subseteq V$ let 
\[
P_S:=\left\{x\in J^{V}~\Big|~
\begin{array}{cc}
x_u\le 1/2 &\textrm{ if } u\in S\\
x_u\ge 1/2 &\textrm{ if } u\notin S
\end{array}\right\}.
\]
Clearly, this defines a decomposition of $J^{V}$ into $2^{k}$ pieces, i.e.
\[
    J^{V}=\bigcup_{S\subseteq V} P_S,
\]
If $S'\neq S$, then there exists $u\in S\setminus S'$ and thus $P_S\cap P_{S'}\subseteq \{x\in J^V~|~x_u=1/2\}$.
Therefore,  we have
\[
    \hat w_T=\sum_{S\subseteq V(T)}\int_{P_S} f(\bm{x}) ~d\mu.
\]
Some of these integrals are actually zero, which we can rule out. 
If $uv\in T$ and $u,v\in S$, then $\int_{P_S} f(\bm{x}) ~d\mu=0$, since it is the volume of a polytope with codimension at least $1$ (since if $f(\bm{x})\neq 0$ then $x_u=x_v=1/2$). Similarly, if $uv\in E_T$ and $u,v\in V\setminus S$, then $\int_{P_S} f(\bm{x})=0$. 
We therefore call a set $S\subseteq V$ \emph{nice} if $S$ is independent in $T$ and $V\setminus S$ is independent in $E_T$.

For a given nice set $S\subseteq V$ endow it with a linear ordering. Let us consider two maps $s,t: V\setminus S\to S\cup \{\oplus,\ominus\}$, such that
\begin{itemize}
    \item if $N_T(u)\cap S= \emptyset$, then $s(u)=\oplus$, otherwise $s(u)\in N_T(u)\cap S$,
    \item if $N_{E_T}(u)\cap S= \emptyset$, then $t(u)=\ominus$, otherwise $t(u)\in N_{E_T}(u)\cap S$.
\end{itemize}
We denote the set of pairs $(s,t)$ of such maps by $\mathcal{F}_S$. Then for each $(s,t)\in\mathcal{F}_S$ we define a polytope
\[
 P_{S,(s,t)}:=\left\{\bm{x}\in P_S~|~\forall u\in V\setminus S: \begin{array}{cc} x_i\geq x_{s(u)}& \text{ for all $i\in N_T(u)\cap S$}\\
 x_i\leq x_{t(u)}& \text{ for all $i\in N_{E_T}(u)\cap S$}
 \end{array}
 \right\},
\]
where we set $x_{\ominus}=1/2-\delta$, $x_{\oplus}=1/2$. 
Clearly, this defines a decomposition of $P_S$ into at most $\Delta^{2k}$ parts, i.e.
\[
    P_S=\bigcup_{(s,t)\in \mathcal{F}_S}P_{S,(s,t)},
\]
which implies that
\[
\int_{P_S} f(\bm{x}) ~d\mu = \sum_{(s,t)\in \mathcal{F}_S} \int_{P_{S,(s,t)}} f(\bm{x}) ~d\mu,
\]
since for distinct $(s,t)$ and $(s',t')$ the intersection of the polytopes $P_{S,(s,t)}$ and $P_{S,(s',t')}$  lies in an affine hyperplane.

Order the elements of $V\setminus S$ as $v_1,\dots,v_\ell$. 
Then for a fixed  $(s,t)\in\mathcal{F}_S$ we have the equality
\begin{equation}\label{eq:int equality}
    \int_{P_{S,(s,t)}}f(\bm{x}) ~d\mu=\int_{J_-^S}\left(\int_{1-x_{s(v_1)}}^{1-x_{t(v_1)}}\dots\int_{1-x_{s(v_\ell)}}^{1-x_{t(v_\ell)}} \bm{1}~d\mu(x_{v_1})\dots ~d\mu(x_{v_\ell})\right)\cdot g_{(s,t)}(\bm{x})~d\mu,
\end{equation}
where $g_{(s,t)}(\bm{x}):\mathbb{R}^S\to\mathbb{R}$ is defined as
\[
    g_{(s,t)}(\bm{x})=\prod_{i=1}^\ell \bm{1}_{x_{s(v_i)}>x_{t(v_i)}}\cdot \prod_{v\in (N_T(v_i)\cap S)\cup\{\oplus \} }\bm{1}_{x_{v}> x_{s(v_i)}}\cdot \prod_{v\in (N_{E_T}(v_i)\cap S)\cup\{\ominus \} }\bm{1}_{x_{v}< x_{t(v_i)}}.
\]
Indeed, this follows from the observation that for a vertex $v\in V\setminus S$ an edge $uv$ with $u\in S$ and we have that if $uv\in T$ the condition $x_u+x_v\geq 1$ translates to $x_v\geq 1-x_u$, while if $uv\in E_T$ the condition $x_u+x_v\leq 1$ translates to $x_v\leq 1-x_u$ (if $v$ has no $T$-neighbor in $S$ we have $s(v)=\oplus$ and the condition $x_v\geq 1/2$ translates to $x_v\geq 1-1/2$ while if $v$ has no $E_T$-neighbor in $S$ we have $t(v)=\ominus$ and the condition $x_v\leq 1/2+\delta$ translates to $x_v\leq 1-(1/2-\delta)$).
). 
Using that for each neighbor $u$ of $v$ we have by definition of the polytope $P_{S,(s,t)}$ that $x_u\geq x_{s(v)}$ and $x_u\leq x_{t(v)}$ the identity~\eqref{eq:int equality} follows.

By~\eqref{eq:int equality} we thus have
\begin{equation}\label{eq: integral_over_poset_polytope}
    \int_{P_{S,(s,t)}}f(\bm{x}) ~d\mu=\int_{J_-^S}\underbrace{\prod_{i=1}^{\ell}(x_{s(v_i)}-x_{t(v_i)})}_{p(\bm{x})}\cdot g_{(s,t)}(\bm{x})~d\mu,
\end{equation}
where $p(\bm{x})$ has total degree at most $k$ and has at most $k$ variables.

Let us next give an interpretation for $g_{(s,t)}(\bm{x})$. 
For a given $(s,t)\in\mathcal{F}_S$ we define the digraph $D=D_{(s,t)}$ on $S$ where $(u,v)\in E(D)$, if there exists $v_i\in V\setminus S$ such that
\begin{itemize}
\item either $v=s(v_i)$ and $u=t(v_i)$,
\item or $u=s(v_i)$ and $v\in N_T(v_i)\cap S$, 
\item or $v=t(v_i)$ and $u\in N_{E_T}(v_i)\cap S$.
\end{itemize}
Observe that each term in $g_{(s,t)}(\bm{x})$ corresponds to an arc of this digraph, in the sense that $0\neq g_{(s,t)}(\bm{x})$ if and only if the ordering defined by $\bm{x}$ on its coordinates is compatible with the digraph $D$, that is, for every $(u,v)\in D$ we have $x_u<x_v$. 
We see immediately that this integral above is $0$ if $D$ contains an oriented cycle. Therefore we may assume that $D=D_{(s,t)}$ is acyclic, i.e. it defines a poset $\mathcal{P}$ on $S$.
The next lemma provides an algorithm for evaluating~\eqref{eq: integral_over_poset_polytope}.

\begin{lemma}
    Let $\mathcal{P}$ a poset on a set $S$ of size $k$ and let $a<b\in\mathbb{R}$. Let $p(\bm{x})$ a polynomial in variables $\{x_u\}_{u\in S}$ of total degree $d$ with integer coefficients. Let 
    \[
    g(\bm{x})=\prod_{u<v} \bm{1}_{x_{u}<x_{v}}\prod_{u\in S} \bm{1}_{b>x_u}\prod_{u\in S} \bm{1}_{x_u>a}.
    \]
    Then
    \[
        \int_{\mathbb{R}^S} p(\bm{x})g(\bm{x}) ~d\mu
    \]
    can be computed in $O(k2^{d+3k})$ time.
\end{lemma}
\begin{proof}
Let us define for each 
$\emptyset\neq U\subseteq S$ a function $F_U:\mathbb{R}^2\times \mathbb{R}^{S\setminus U}\to \mathbb{R}$ by
\[
    F_U(m,M,\{x_v\}_{v\in S\setminus U}):=\sum_{u\in \min U} \int_{m}^M F_{U-u}(m,M,\{x_v\}_{v\in S\setminus U\cup \{u\}})|_{m=x_u} ~dx_u,
\]
where $\min U$ denotes the set of minimal elements of the poset restricted to $U$, and define
\[
    F_{\emptyset}(m,M,\{x_v\}_{v\in S}):=p(\bm{x}).
\]
It is clear that each $F_U$ is polynomial in the $k+2-|U|$ variables $\{x_{v}\}_{v\notin U}$, $m$ and $M$ and it has total degree at most $d+|U|$. 
This means that among all the polynomials $\{F_U\}_{U\subseteq S}$ we see at most $\binom{d+2k+2}{k+2}< 2^{d+2k+2}$ many different monomials. 
This means that to calculate $F_U$ from $F_{U-u}$ for a minimal element $u$ of $U$ (by which we mean computing the coefficients of each monomial) we need at most $O(k2^{d+2k})$ time.
This implies that to calculate $F_S$ we need at most $O(k2^{d+2k}\cdot 2^k)=O(k2^{d+3k})$ time. 

We next claim that 
\begin{equation}\label{eq:claim FU}
    F_U(m,M,\{x_u\}_{u\in S\setminus U})|_{m=a,M=b}=\int_{\mathbb{R}^U} p(\bm{x}) g_{U}(a,b,\bm{x}) ~d\mu,
\end{equation}
where 
\[
g_U(m,M,\bm{x})=\prod_{u,v \in U, u<v} \bm{1}_{x_u<x_v}\prod_{u\in U} \bm{1}_{x_u<M}\prod_{u\in U} \bm{1}_{x_u>m}.
\]
We will prove this by induction on the size of $U$. If $U=\emptyset$, then it is trivial, so let us assume that $U\neq \emptyset$. Then using that we can partition the space by saying which minimal element of $U$ is the smallest we have 
\begin{align}
    \int_{\mathbb{R}^U} p(\bm{x}) g_{U}(a,b,\bm{x}) ~d\mu&=\int_{\mathbb{R}^U}\sum_{u\in \min U} \left(\prod_{\substack{v\in \min U\\v\neq u }}\bm{1}_{x_v>x_u}\right) p(\bm{x}) g_{U}(a,b,\bm{x}) ~d\mu \nonumber\\
    &=\sum_{u\in \min U} \int_{\mathbb{R}^U}\left(\prod_{\substack{v\in \min U\\v\neq u }}\bm{1}_{x_v>x_u}\right) p(\bm{x}) g_{U}(a,b,\bm{x}) ~d\mu \nonumber\\
    &=\sum_{u\in \min U} \int_{\mathbb{R}^U}\left(\prod_{\substack{v\in \min U-u}}\bm{1}_{x_v>x_u}\right) p(\bm{x}) g_{U-u}(a,b,\bm{x})\cdot \bm{1}_{x_u>a}\bm{1}_{x_u<b} ~d\mu \nonumber\\
    &=\sum_{u\in \min U} \int_{\mathbb{R}^U} p(\bm{x}) g_{U-u}(x_u,b,\bm{x})\cdot \bm{1}_{x_u>a}\bm{1}_{x_u<b} ~d\mu. \label{eq:before fubini}
    \end{align}
Next we apply Fubini's theorem to obtain that~\eqref{eq:before fubini} is equal to
    \begin{align*}
 &=\sum_{u\in \min U} \int_{a}^b\left(\int_{\mathbb{R}^{U-u}} p(\bm{x}) g_{U-u}(x_u,b,\bm{x}) ~d\mu\right)~dx_u\\
    &=\sum_{u\in \min U} \int_{a}^bF_{U-u}(x_u,b,\{x_v\}_{v\in S-U\cup\{u\}})~dx_u,
\end{align*}
where the last step follows by using induction.
We thus see that  $ \int_{\mathbb{R}^U} p(\bm{x}) g_{U}(a,b,\bm{x}) ~d\mu$ equals $F_U(m,M,\{x_u\}_{u\in S-U})|_{m=a,M=b}$ by definition, thereby proving~\eqref{eq:claim FU}.

As a consequence we obtain that for any $a<b$ we have
\[
    \int_{\mathbb{R}^S} p(\bm{x})g(\bm{x})~d\mu =F_S(a,b).
\]
This finishes the proof, as we have already indicated how to compute the $F_U$ for $U\subseteq S$.
\end{proof}

\begin{proof}[Proof of Proposition~\ref{prop:volume}]
Using the notations above we see that
\[
\hat w_T=\sum_{\substack{S\subseteq V(T)\\ \text{$S$ is nice}}}\left(\sum_{\substack{(s,t)\in \mathcal{F}_S\\ \text{$D_{(s,t)}$ is acyclic}}} \int_{P_{S,(b,t)}} f(\bm{x})~d\mu\right).
\]
By the previous lemma we can compute
\[
\int_{P_{S,(s,t)}} f(\bm{x})~d\mu
\]
 in $O(k2^{4k})$ time, thus $\hat w_T$ can be computed in $O(2^k\cdot \Delta^{2k} \cdot 2^{4k}\cdot \textsc{poly}(k))=\Delta^{O(k)}$  time.


\end{proof}
\subsection{Computation of $\lambda_{H,k}$}
We now conclude this section by giving a proof of Theorem~\ref{thm:algorithm}.

\begin{proof}[Proof of Theorem~\ref{thm:algorithm}]
We may assume that the graph $H$ has maximum degree at most $\Delta$ for otherwise the coefficient of $\textrm{ind}(H,\cdot)$ in $P_{G,\delta}$ will be zero.
Since we can also assume that $H$ has at most $2k$ vertices, it will have at most $\Delta k$ edges.
By~\eqref{eq:coef ind} we need to enumerate all forests $F$ of $k$ edges in $H$ and compute the weights $w_T$ of each component $T$ of $F$.
To enumerate all forests with $k$ edges in $H$ we just go over all subsets of $E(H)$ with $k$ elements, which takes time bounded by $\binom{\Delta k}{k}\le (e\Delta)^k$, by Stirling's formula.  
For each such a set we identify the components, $T_1,\ldots, T_\ell$ and whether these components are trees with at least one edge, which takes $\textsc{poly}(k)$ time.

    By the algorithm of Proposition~\ref{prop:volume} we compute the normalized weights 
    \[
    \hat{w}_{T_i}=(-1)^{|T_i|}(1/2+\delta)^{|V(T_i)|}w_{T_i}
    \] and hence the weights $w_{T_i}$ in $\Delta^{O(|V(T_i)|)}$ time. 
All together the algorithm takes $\Delta^{O(k)}$ time.
\end{proof}



\section{Concluding remarks, open problems and discussion}~\label{sec:remarks}
We discuss here some extensions and limitations of our method and mention some open problems. 

Gamarnik and Smedira~\cite{gamarnik2023computing} indicate that their method for approximating the volume of $P_{G,\delta}$ does not extend all the way to $\delta=1/2$.
Even though we improve on their bound on $\delta$ from $O(1/\Delta^2)$ to $O(1/\Delta)$, our method also comes with limitations that we will explain now.
As detailed in~\eqref{eq:rewrite to forest} the volume of the polytope $P_{G,\delta}$ can be seen as the evaluation of a multivariate forest generating function, i.e. the volume can be expressed as 
\[
    \tilde p_G(x;\{w_T\}):=\sum_{\substack{F\subseteq E(G)\\ \textrm{$F$ forest}}} \prod_{\textrm{$T$ comp. of $F$}}w_T\cdot x^{|F|},
\]
where $\tilde p_G(x;\{w_T\})$ is the multivariate forest generating function with tree weights $\{w_T\}$.
Crucial for establishing absence of zeros of $p_{G,\delta}$ is that the tree weights $|w_T|$ are exponentially small (in our setting they are bounded by $(4\delta)^{|V(T)|}$, see Lemma~\ref{lem:weight bound}). This is in fact the only property that we use, in the sense that there is a $\delta'=O(1/\Delta)$ such that if for every tree $T$ the absolute value of its weight is in $[0,\delta'^{|E(T)|}]$ then for any graph of maximum degree at most $\Delta$ the polynomial $\tilde p_G(x;\{\tilde w_T\})$ has  a zero-free disk of radius bigger than 1.
In particular, it means that 
\[
    \tilde{p}_{G,\delta'}(x):=\sum_{\substack{F\subseteq E(G)\\ \textrm{$F$ forest}}}(\delta'x)^{|F|}
\]
which is the forest generating function of $G$ evaluated at $\delta'x$, has a zero-free disk of radius larger than 1 (for the variable $x$). A result of the first author together with Csikv\'ari~\cite{BCtutteregular} implies that no zero-free disk of radius larger than $1$ (for the variable $x$) exists when $\delta'>\frac{1}{\Delta-2}$. 
This suggests that to extend Theorem~\ref{thm: zero-free disk} beyond $\delta=O(1/\Delta)$ one needs to devise a proof that uses more information about the weights $w_T$ than just the fact they are exponentially small.
We leave it as a challenging open problem to do so, or to find another approach for deterministically approximating the volume of $P_{G,\delta}$ for $\delta\gg 1/\Delta$, remarking that the results from~\cite{GamarnikKavitacontinuous} do suggest that this should be possible.

\quad \\
Now let us turn to possible extensions of our approach. Our approach relies on three ingredients: 1) express the volume of the polytope as an evaluation of a graph polynomial, 2) establish a zero-free region for this polynomial and 3) `efficiently' compute the low (logarithmically small) order terms of the Taylor series.
\begin{itemize}
    \item \textbf{Volumes of polytopes} Let $P$ be a polytope defined by $A\bm{x}\le b$ and  $\bm{x}\in [0,1]^d$, where $A\in\mathbb{R}^{n\times d}$. If each row of $A$ has at most $2$ non-zero coefficient, then in a similar fashion we can express the volume of $P$ as an evaluation of the multivariate  forest generating  polynomial.  To address the case when the number of non-zero coefficients in each row of $A$ is bounded by some $k>2$, we obtain the volume as an evaluation of hypergraph polynomial. We expect that the ideas of \cite{fadnavis2015roots} are applicable in this scenario to establish a zero-free region for a similarly defined polynomial when the polytope is appropriately truncated.
    
    To compute the low order terms of the Taylor series of the logarithm one could use the formula of Lawrence~\cite{Lawrencevolume}. It is not clear to us how to diretly generalize our combinatorial approach.

    \item \textbf{Homomorphism count} Let $H$ be a weighted graph on $[m]=\{1,\ldots,m\}$ with adjacency matrix $A$. Then the number of homomorphism of $G$ into $H$ can be expressed as
    \begin{equation}\label{eq:homomorphism count}
        \mathrm{hom}(G,H)=\sum_{\phi: V(G)\to [m]} \prod_{uv\in E(G)} A_{\phi(u),\phi(v)}.
    \end{equation}
    We can think about this expression as a discretized generalization of \eqref{eq:vol with indicator}. Indeed, if we let $I$ be the interval $[0,m]$ and let $f:I\times I\to\mathbb{R}$ be the step function representing $A$, that is
    \[
        f(x,y)=\left\{\begin{array}{cc}
        A_{\lfloor x/m\rfloor,\lfloor y/m\rfloor} & \textrm{if $x,y\in I$},\\
        0 &\textrm{otherwise},
        \end{array}
        \right.
    \]
    then
    \begin{equation}\label{eq:hom with integral}
    \mathrm{hom}(G,H)=\int_{I^{V(G)}}\prod_{uv\in E(G)}f(x_u,y_v)~d\mu.
    \end{equation}
    The approximation of $\mathrm{hom}(G,H)$ has already been addressed in case all the entries of $A$ are close to $1$~\cites{barbook,PR17}. 
    It is proved that for any $\Delta,m$ there is a algorithm to give $\varepsilon$-approximation of $\mathrm{hom}(G,H)$ for graphs of maximum degree $\Delta$ (assuming that the entries of $A$ are $0.35/\Delta$ close to $1$) in polynomial time in $|V(G)|/\varepsilon$, however the degree of this polynomial depends on $m$. 
    Our current approach allows to also treat matrices $A$ for which not all of its entries that are close to $1$, see our next comment below.

    \item \textbf{Homomorphism densities and general integrals} One could further generalize \eqref{eq:hom with integral} by simply allowing $f$ to be an arbitrary measurable (symmetric) function on some measure space $(X,\mu)^2$. In particular, if $f$ is $[0,1]$ valued and $(X,\mu)$ is a probability space, then  $f$ is a graphon  and \eqref{eq:hom with integral} expresses the homomorphism density of $G$ in the graphon $f$ (see \cite{lovasz2012large}). For this setup let us remark the following observations.

    Similarly as in Section 2 we have 
    \[
        \mathrm{hom}(G,f)=p_G(x)\big|_{x=1}=\sum_{k\ge 0}\sum_{P\in\pi_{n-k}(V)}\prod_{S\in P}w(G[S])x^k \big|_{x=1},
    \]
    where
    \[
        w(H)=\sum_{\substack{T\subseteq E\\ (V(H),T) \text{ tree}}}(-1)^{|T|}\underbrace{\int_{X^S}\prod_{uv\in T}(1-f(x_u,x_v))\cdot \prod_{st\in E(S)_T}f(x_s,x_t) d\mu}_{w_T}.
    \]
    For a given graphon $f$ let $c(f)=\sup_{x\in X} \int_X (1-f(x,y)) ~dy$. This parameter plays a crucial role in~\cite{borgs2013left} to establish the convergence radius of a cluster expansion for the homomorphism count in the context of sparse graph limits. 
    It is not hard to see that $|w_T|\le c(f)^{|T|}$, thus one can easily obtain that if $c(f)$ is at most $O(1/\Delta)$, then $p_G(x)$ is zero-free on a disk of radius bigger than $1$ (i.e. a variant of Theorem~\ref{thm: zero-free disk} holds).    
    However, it is not clear how to compute the low order terms of the polynomial in an efficient way for an arbitrary function $f$.
    Our main result (Theorem~\ref{thm:main}) deals with functions of the form $f=\mathbf{1}_{x+y\le 1}$ and we heavily relied on the combinatorics provided by $f$.
    
    A different approach to establish an approximation algorithm for computing the homomorphism density in $f$ would be by to approximate $f$ with a step-function $\hat f$. 
    This idea was used in~\cite{gamarnik2024integrating} in combination with the correlation decay method to find a deterministic quasi-polynomial time approximation algorithm for computing $\mathrm{hom}(G,f)$ (provided that $f$ is smooth enough and close enough to a constant function). 
    Combing this idea with our approach we can also obtain a quasi-polynomial time approximation algorithm, but we have to leave it as an open problem to improve the running time to actual polynomial time.
 \end{itemize}

\bibliographystyle{plain}
\bibliography{volume}
\end{document}